\documentclass[11pt]{amsart}
\usepackage{graphicx}
\usepackage{amssymb}
\usepackage[usenames,dvipsnames]{color}
\numberwithin{equation}{section} 

\newtheorem{theorem}{Theorem}[section]
\newtheorem{lemma}[theorem]{Lemma}

\newtheorem{corollary}[theorem]{Corollary}

\newtheorem{remark}[theorem]{Remark}
\newtheorem{definition}{Definition}[section]

\def\R{{\mathbb R}}

\def\cL{{\mathcal L}}
\def\cM{{\mathcal M}}

\def\a{\alpha}
\def\b{\beta}
\def\e{\varepsilon}
\def\d{\delta}
\def\D{\Delta}
\def\g{\gamma}
\def\G{\Gamma}
\def\k{\kappa}
\def\l{\lambda}
\def\L{\Lambda}

\def\p{\partial}
\def\r{\rho}
\def\s{\sigma}
\def\t{\tau}
\def\w{\omega}
\def\W{\Omega}

\def\1{\left(}
\def\2{\right)}
\def\3{\left\{}
\def\4{\right\}}
\def\8{\infty}
\def\sm{\setminus}
\def\ss{\subseteq}

\DeclareMathOperator*{\supp}{supp}
\DeclareMathOperator*{\dist}{dist}

\title{Regularity for solutions of non local parabolic equations II}
\author[H. Chang Lara]{H\'ector Chang Lara}
\address{%
University of Texas at Austin\\
Department of Mathematics\\
1 University Station C1200\\
Austin, TX 78712
}
\email{hchang@math.utexas.edu}
\author[G. D\'avila]{Gonzalo D\'avila}
\address{%
University of British Columbia\\
Department of Mathematics\\
1984 Mathematics Road\\
Vancouver, B.C. Canada V6T 1Z2 
}
\email{gdavila@math.ubc.ca}

\begin{document}
\begin{abstract}
We prove boundary regularity and a compactness result for parabolic nonlocal equations of the form $u_t-Iu=f$, where the operator $I$ is not necessarily translation invariant. As a consequence of this and the regularity results for translation invariant case, we obtain $C^{1,\a}$ interior estimates in space for non translation invariant operators under some hypothesis on the time regularity of the boundary data.
\end{abstract}

\maketitle


\section{Introduction}
In this work we are interested in studying regularity of solutions of 
\begin{align}\label{eqintro}
u_t(x,t)-Iu(x,t)=f(x) \text{ in } B_1\times(-1,0],
\end{align}
where $I$ may be given by
\begin{align*}
Iu(x,t)=\inf\limits_\a\sup\limits_\b\int_{\R^n}(u(x+y,t)+u(x-y,t)-2u(x,t))K_{\a,\b}(x,t;y)dy,
\end{align*}
and $K_{\a,\b}(x,t;\cdot)$ is a kernel comparable to the kernel of the fractional laplacian of order $\s\in(0,2)$. Equations of the form \eqref{eqintro} appear naturally when studying the evolution of a jump process and stochastic control problems. 

The study of the regularity of solutions to \eqref{eqintro} for the fully non linear case started in \cite{CD}, where the authors proved H\"older estimates in the non translation invariant setting. In the case of the translation invariant case, i.e., $f=0$, $K(x,t;y)=K(y)$ the authors prove, under some suitable hypothesis on the kernels, that the spatial gradient of the solutions of the equation \eqref{eqintro} had interior H\"older estimates. 

The variational problem was studied by L. Caffarelli, C. Chan and A. Vasseur in \cite{CHV} by using De Giorgi's technique. Also in a recent work from M. Felsinger and M. Kassmann \cite{FK} they prove a Harnack inequality in the divergence case, using Moser's technique. As consequence they obtain H\"older interior regularity. Moreover, their result is stable under the order of the equation, so they recover the local case as a limit.

In the elliptic case, L. Caffarelli and L. Silvestre studied in \cite{C2} regularity of solutions for non translation invariant operators. They prove that if an operator $I_1$ is close to an operator $I_2$, and $I_2$ is a translation invariant operator with $C^{1,\a}$ a priori estimates, then solutions of $I_1u=f$ also have $C^{1,\tilde \a}$ estimates for some $\tilde\a\leq\a$. The method relies on a perturbative argument and a compactness result. As a consequence they recover classical Cordes-Nirenberg type of estimates under some extra assumptions on the operator.

In this work we extend the results of \cite{C2} to the parabolic case. The main theorem (Theorem \ref{C1a theorem}) states that solutions of \eqref{eqintro}, with some regularity assumptions in the time variable for the boundary data, have interior $C^{1,\tilde \a}$ estimates in space, if $I$ is close to a translation invariant operator with interior $C^{1,\a}$ estimates and $\|f\|_\8$ is small enough. We also prove further regularity in time under a different set of hypothesis. 

The paper is divided as follows. In Section \ref{VSP} we introduce the family of operators we are considering, the notion of viscosity solution, introduce the norm for the operators $I$ and recall some previous results from \cite{CD}. In Section \ref{SBR} we prove boundary regularity for solutions of \eqref{eqintro}. To achieve this we use the barriers from \cite{CD} and some careful estimates. In Section \ref{SFRT}, we prove that, under some extra hypothesis on the boundary data, solutions of \eqref{eqintro} have $C^{1,\a}$ estimates in time, even for rough kernels. In Section \ref{SSC} stability and compactness results are proven. This section is key in order to proceed with the perturbative method from Section \ref{C1a section}. Finally, in Section \ref{C1a section} we prove our main result. We also present some applications to linear and non linear equations with variable coefficients.

\section{Preliminaries}\label{VSP}

The (spatial) linear operators $L_K$ we are interested are of the form,
\begin{align*}
L_Ku(x,t)&=\int\d(u,x,t;y)K(x,t;y)dy,\\
\d(u,x,t;y)&=u(x+y,t)+u(x-y,t)-2u(x,t)
\end{align*}
with $K$ non negative, even in $y$ and satisfying the following integrability condition,
\begin{align}\label{intcond}
\int K(x,t;y)\min(1,|y|^2)dy<\infty.
\end{align}

A sufficient regularity and integrability that we need to ask to the function $u$ in order to make the integral above convergent is $u(\cdot,t) \in C^{1,1}(x) \cap L^\8(\R^n)$. The functional space $C^{1,1}$ consists of all the functions for which there exists a vector $p\in\R^n$ such that $\limsup_{|y|\to0}|u(x+y)-u(x)-p\cdot y|/|y|^2 < \8$ and in particular $\limsup_{|y|\to0}|\d(u,x,t;y)|/|y|^2 < \8$.

Most of the time we will consider kernels $K(x,t;y)$ that are controlled by a weight $\w \in L^1$ away from the singularity. This means that
\begin{align*}
\sup_{\R^n\sm B_r} \frac{K(x,t;y)}{\w(y)}<\8 \text{ for every radius $r>0$}.
\end{align*}
In this case the boundedness of $u$ can be relaxed for $u$ being in $L^1(\w)$ which consist of all the functions integrable against $\w$.


\subsection{Basic properties of the operators}

Combinations by inf-sup operations of the linear operators introduced above satisfy several important properties which we resume in the following Lemma.

\begin{lemma}\label{defisok}
Let $\{K_{\a,\b}\}_{(\a,\b)\in A\times B}$ be a family of kernels satisfying the following hypothesis:
\begin{enumerate}
\item Uniform integrability condition:
\begin{align*}
\sup_{\substack{\a,\b\in A\times B,\\(x,t)\in\W\times(-T,0]}} \int K_{\a,\b}(x,t;y)\min(1,|y|^2)dy < \8.
\end{align*}
\item Equicontinuity: For every $(x_0,t_0) \in \W\times(-T,0]$,
\begin{align*}
\lim_{(x,t)\to(x_0,t_0^-))}\sup_{\a,\b\in A\times B}\left|\int\1K_{\a,\b}(x,t;y)-K_{\a,\b}(x_0,t_0;y)\2\min(1,|y|^2)dy \right| = 0
\end{align*}
\item There exists a weight $\w \in L^1$ such that the kernels are uniformly controlled by $\w(y)$ away from the singularity,
\begin{align*}
\sup_{\substack{\a,\b\in A\times B,\\(x,t)\in\W\times(-T,0]\\y\in\R^n\sm B_r}} \frac{K_{\a,\b}(x,t;y)}{\w(y)} < \8 \text{ for every radius $r>0$}. 
\end{align*}
\item For every $x\in\R^n$,
\begin{align*}
\sup_{y\in\R^n}\frac{\w(y-x)}{\w(y)} < \8.
\end{align*}
\end{enumerate}
Let $I$ be defined as the following inf-sup combination of linear operators:
\begin{align*}
Iu(x,t) = \inf_{\a\in A}\sup_{\b\in B} L_{K_{\a,\b}}u(x,t) = \inf_\a\sup_\b \int \d(u,x,t;y)K_{\a,\b}(y)dy,
\end{align*}
Then $I$ is a continuous operator $\W\times(-T,0]$ elliptic with respect to $\{L_{K_{\a,\b}}\}$. Moreover, if each $K_{\a,\b}(x,t;y) = K_{\a,\b}(y)$ is independent of $(x,t) \in \W\times(-T,0]$ then $I$ is also translation invariant.
\end{lemma}

%

The last condition on $\w$ is important to bound $\|u(x+\cdot,t)\|_{L^1(\w)}$ given that $\|u(\cdot,t)\|_{L^1(\w)}<\8$. To be precise,
\begin{align*}
\|u(x+\cdot,t)\|_{L^1(\w)} &= \int |u(x+y,t)|w(y)dy,\\
&= \int |u(y,t)|w(y-x)dy,\\
&\leq \1\sup_{y\in\R^n}\frac{\w(y-x)}{\w(y)}\2 \|u\|_{L^1(\w)} < \8.
\end{align*}

The proof of the Lemma will be given after we define precisely what does it mean to be continuous, translation invariant and elliptic.


\subsubsection{Continuous operators}

Classically we understand that an operator is continuous if, when tested against a smooth test function, returns a continuos function. For the non local setting we need to consider not only functions which are locally smooth but also with tails such that their integrals are continuous. The continuity in space is expected from because the operator is computed by a convolution integral however, the continuity in time has to be imposed in some way.

\begin{definition}
Let $C(a,b;L^1(\w))$ be the space of function $u:(a,b]\to\R$ such that
\begin{enumerate}
\item[(i)] for every $t\in(a,b]$, $u(\cdot,t)\in L^1(\w)$,
\item[(ii)] for every $t_2 \in (a,b]$, $\|u(\cdot,t_1)-u(\cdot,t_2)\|_{L^1(\w)} \to 0$ as $t_1\to t_2^-$.
\end{enumerate}
It comes additionally with the norm,
\begin{align*}
\|u\|_{C(a,b;L^1(\w))} = \sup_{t\in(a,b]}\|u(\cdot,t)\|_{L^1(\w)}.
\end{align*}
\end{definition}

The space $C(a,b;L^1(\w))$ with the topology given by its norm is a separable Banach space. This will be used in the compactness arguments in Section \ref{SSC}.

The space of functions against which we test the continuity of $I$ are given by parabolic second order polynomials and functions in $C(a,b;L^1(\w))$.

\begin{definition}
The space $S = S(\W\times(-T,0])$ of test functions is the set of all pairs $(v,B_r(x)\times(t-\t,t])$ such that $v \in C(t-\t,t;L^1(\w))$, $B_r(x)\times(t-\t,t] \ss \W\times(-T,0]$ and $v$ restricted to $B_r(x)\times(t-\t,t]$ is a quadratic parabolic polynomial, i.e.
\begin{align*}
v(x,t) = \sum_{i,j=1}^n a_{i,j}x_ix_j + \sum_{i=1}^n b_ix_i + ct + d. 
\end{align*}
\end{definition}

We say that a non local operator $I$ is continuous if for every $(v,B_r(x)\times(t-\t,t]) \in S$, we have that $Iu$ is a continuous function in $B_r(x)\times(t-\t,t]$ (with respect to the parabolic topology).

Notice that without the assumption coming from $C(a,b;L^1(\w))$ even the fractional laplacian can give us troubles in terms of continuity. Take for example $u$ equal to zero in $B_1\times(-1,0]$ and let vary $u$ freely outside $B_1\times(-1,0]$ to get a discontinuous function in the time variable, $f = u_t- \Delta^{\sigma/2} u$.


\subsubsection{Translation invariant operators}

In order to introduce the translation invariant operators we consider the shift $\t_{(x,t)}$ acting over a function $u$,
\begin{align*}
(\t_{(x,t)}u)(y,s) = u(y+x,s+t). 
\end{align*}
We say that an operator $I$ is translation invariant if for every $(y,s) \in \R^n\times\R$ and every $(x,t) \in \R^n\times\R$ where $I$ can be evaluated, we have then that $I$ can also be evaluated for $\t_{(x-y,t-s)}u$ at $(y,s)$ and
\begin{align*}
I(\t_{(x-y,t-s)}u)(y,s) = Iu(x,t).
\end{align*}

Whenever $I$ is not necessarily translation invariant we may also write $I(u,x,t)(y,s)$ when we want to ``freeze the coefficients of $I$ at $(x,t)$'' and evaluate at a given function $u$,
\begin{align*}
I(u,x,t)(y,s) := I(\t_{(y-x,s-t)}u)(x,t).
\end{align*}
When $I$ is translation invariant $I(u,x,t)(y,s) = Iu(y,s)$ always holds. Also $J = I(\cdot,x_0,t_0)$ is a translation invariant operator because,
\begin{align*}
J(\t_{(x-y,t-y)}u)(y,s) &= I(\t_{(x-y,t-s)}u,x_0,t_0)(y,s) = I(\t_{(x-x_0,t-t_0)}u)(x_0,t_0),
\end{align*}
and also
\begin{align*}
Ju(x,t) = I(u,x_0,t_0)(x,t) = I(\t_{(x-x_0,t-t_0)}u)(x_0,t_0).
\end{align*}


\subsubsection{Ellipticity.}

The notion of ellipticity traditionally refers to some positivity condition on the operator. In the definition given in \cite{C1} this is achieved by controlling the non linear operators by a family of linear ones which already have the positivity condition established as a lower bound for the kernels. We give next the exact same definition however the positivity requirement does not appear at this moment. Later on we will also impose a positive lower bound to the family of kernels controlling $I$.

We say that a fully non linear operator $I$ is elliptic with respect to a class $\cL$ of linear operators if for every $x$ in the domain of $I$ ($t$ is fixed and ommited),
\begin{align}\label{ellipticity}
\cM^-_{\cL}(u-v)(x) \leq Iu(x)-Iv(x) \leq \cM^+_{\cL}(u-v)(x),
\end{align}
where the extremal operators $\cM^\pm_\cL$ are defined in the following way,
\begin{align*}
\cM_\cL^+w(x) &= \sup_{L\in\cL}Lu(x),\\
\cM_\cL^-w(x) &= \inf_{L\in\cL}Lu(x).
\end{align*}


\subsubsection{Proof of Lemma \ref{defisok}}

Let $\a$ and $\b$ be fixed. For $u(\cdot,t) \in C^{1,1}(x)\cap L^1(\w)$ there is some radius $r$ and a positive constant $M$ such that 
\[
|\d(u,x,t;y)| \leq M|y|^2,\ \hbox{in}\ B_r.
\]
This allows us to estimate,\begin{align*}
\int|\d(u,x,t;y)K_{\a,\b}| &= \int_{B_r}M|y|^2 K_{\a,\b}dy\\
&{} + C\int_{\R^n\sm B_r}|u(x+y,t)+u(x-y,t)-2u(x,t)|\w dy,\\
&\leq C(M+\|u(\cdot,t)\|_{L^1(\w)} + |u(x,t)|).
\end{align*}
This implies that $L_{\a,\b}u(x,t)$ is well defined and bounded, uniformly in $(\a,\b)$, therefore $Iu(x,t)$ is well defined.

It is immediate from the definition that $I$ is elliptic with respect to $\cL = \{L_{K_{\a,\b}}\}$ and is translation invariant if each kernel $K_{\a,\b}(x,t;y) = K_{\a,\b}(y)$ is independent of $(x,t) \in \W\times(-T,0]$.

To check continuity we need to test $I$ against test functions $(v,B_r(x_0)\times(t_0-\t,t_0]) \in S$. 
It would follow if we show that for each $(\a,\b)\in A\times B$ we have that the following limit holds uniformly in $(\a,\b)$,
\begin{align*}
\lim_{(x,t)\to(x_0,t_0^-)} L_{K_{\a,\b}}v(x,t) = L_{K_{\a,\b}}v(x_0,t_0).
\end{align*}
Let $(x_1,t_1)\in B_r(x_0)\times(t_0-\t,t_0]$ and denote,
\begin{align*}
\d_0(y) &= \d(v,x_0,t_0;y),\\
\d_1(y) &= \d(v,x_1,t_1;y),\\
K_0(y) &= K_{\a,\b}(x_0,t_0;y),\\
K_1(y) &= K_{\a,\b}(x_1,t_1;y).
\end{align*}
By the Leibnitz formula,
\begin{align*}
L_{K_{\a,\b}}v(x_1,t_1) - L_{K_{\a,\b}}v(x_0,t_0) &= \int (\d_1-\d_0)K_1 + \d_0(K_1-K_0)dy.
\end{align*}
The second term in the integral goes to zero, uniformly in $(\a,\b)$, because of the equicontinuity of the kernels. For the first term we do it in two steps, if $t_1 = t_0$ and if $x_1 = x_0$. The whole equicontinuity follows then by the triangular inequality.

Consider $t_1 = t_0$. In this case we can just forget about the time variable which has been fixed. Since $L_{K_1}$ is now translation invariant the continuity in this case follows because then the definition is equivalent to the one in \cite{C1}.

Consider now $x_1=x_0$. We use the fact that $v\in C(t_0-\t,t_0,L^1(\w))$
\begin{align*}
\left|\int (\d_1-\d_0)K_1\right| &\leq C(\|v(\cdot+x,t_1)-v(\cdot+x,t_0)\|_{L^1(\w)}\\
&{} + \|v(\cdot-x,t)-v(\cdot-x,s)\|_{L^1(\w)}\\
&{} + \|v(\cdot,t)-v(\cdot,s)\|_{L^1(\w)}),\\
&\leq C\|v(\cdot,t)-v(\cdot,s)\|_{L^1(\w)}.
\end{align*}
It then goes to zero independently of $(\a,\b)$. This concludes the equicontinuity and the proof of the Lemma.

\subsection{Sufficient hypothesis in our proofs}

One important feature of the classical regularity theory of non divergence equations is that the estimates hold at every scale by having an scale invariance property. Next, we briefly discuss what does this means in our case. 

\subsubsection{Scaling}

Consider a smooth bounded function $u$ and a operator $I$ such that
\begin{align*}
u_t - Iu = f \text{ in $\W\times(-T,0]$}.
\end{align*}
If we rescale $u$ by $u_{\a,\b,\g}(x,t) = \a u(\b x,\g t)$ then the equation gets rescaled in the following way,
\begin{align*}
(u_{\a,\b,\g})_t - I_{\a,\b,\g}u_{\a,\b,\g} = f_{\a,\b,\g} \text{ in $\b^{-1}\W\times[-\g^{-1}T,0]$}, 
\end{align*}
where
\begin{align*}
(I_{\a,\b,\g}v)(x,t) &= \g\a I(\a^{-1}v(\b^{-1}\cdot,\g^{-1}\cdot))(\b x,\g t),\\
f_{\a,\b,\g}(x,t) &= \g\a f(\b x,\g t).
\end{align*}

If $I=L$ is linear with kernel $K$ then the equation for $u_{\a,\b,\g}$ is also linear and can be obtained by the change of variable formula. The kernel $K_{\a,\b,\g}$ for $L_{\a,\b,\g}$ is given by
\begin{align*}
K_{\a,\b,\g}(x,t;y) = \g\b^n K(\b x,\g t; \b y).
\end{align*}


\subsubsection{Operators comparable to the fractional laplacian.}

There will be two important families of linear operators we will frequently use. They are denoted by $\cL_0(\s,\L) \supseteq \cL_1(\s,\L)$ and depend on parameters $\s \in (0,2)$ and $\L \geq 1$. For all of them, $\w$ will be fixed to be $w(y)=1/(1+|y|^{n+\s})$ or $w(y)=1/(1+|y|^{n+\s_0})$ for some $0 < \s_0 \leq \s$ if we want to consider all the equations of order between $\s_0$ and $2$.

$\cL_0$ consists of all the linear, translation invariant, operators $L$ such that their kernels $K$ are comparable to the kernel for the laplacian of order $\s$,
\begin{align}
(2-\s)\frac{\L^{-1}}{|y|^{n+\s}}\leq K(y) \leq (2-\s)\frac{\L}{|y|^{n+\s}}.
\end{align}
In this family the extremal operators take the explicit form
\begin{align*}
\cM^+_{\cL_0} v(x,t) &:= \sup_{L\in\cL_0}(Lv)(x,t)\\
& = (2-\s)\int\limits_{\R^n}\frac{\L\d^+(v,x,t;y)-\L^{-1}\d^-(v,x,t;y)}{|y|^{n+\s}}dy,
\end{align*}
and 
\begin{align*}
\cM^-_{\cL_0} v(x,t) &:= \inf_{L\in\cL_0}(Lv)(x,t)\\
& = (2-\s)\int\limits_{\R^n}\frac{\L^{-1}\d^+(v,x,t;y)-\L\d^-(v,x,t;y)}{|y|^{n+\s}}dy.
\end{align*}
Here $\d^\pm$ denotes the positive and negative part of $\d$ ($\d = \d^+ - \d^-$).

Notice that in this case the hypothesis of Lemma \ref{defisok} are immediately satisfied.

The factor $(2-\s)$ becomes important as $\s \to 2^-$ as they will allow us to recover second order differential operators.

The family $\cL_1$ satisfy additionally that for every kernel $K$,
\begin{align*}
|DK(y)| \leq \L|y|^{-(n+\s+1)}.
\end{align*}

An important property of these families is that the parabolic equations associated with them remain invariant under scaling. If $I$ is elliptic with respect to $\cL$ and $u$ is a smooth function such that
\begin{align*}
u_t - Iu = f \text{ in } \W\times(-T,0],
\end{align*}
then $u_\b = u(\b x, \b^\s t)$ satisfies
\begin{align*}
(u_\b)_t - I_\b u_\b = f_\b \text{ in } \b^{-1} \W\times(-\b^{-\s} T,0],
\end{align*}
where,
\begin{align*}
I_\b v(x,t) &= \b^\s I(v(\b^{-1}\cdot,\b^{-\s}\cdot))(\b x, \b^\s t),\\
f_\b(x,t) &= \b^\s f(\b x, \b^\s t).
\end{align*}
Therefore, the ellipticity with respect to $\cL$ gets transformed, for $I_\b$, into ellipticity with respect to $\cL_\b$ which consist of all the linear operators $L_\b$ obtained from $L\in\cL$ according to $L_\b v(x,t) = \b^\s L(v(\b^{-1}\cdot,\b^{-\s}\cdot))(\b x, \b^\s t)$. If $L$ has kernel $K$, then $L_\b$ has kernel $K_\b(x,t) = \b^{n+\s}K(\b x,\b^\s t)$. In particular, if $\cL = \cL_i$ then $\cL_\b = \cL$.


\subsection{Viscosity solutions.}

Viscosity solutions always assume a minimum requirement of continuity. Here we denote the space of upper semicontinuous functions in $\bar\Omega\times[-T,0]$, always with respect to the parabolic topology, by $USC(\bar\Omega\times[-T,0])$. Similarly, $LSC(\bar\Omega\times[-T,0])$ denotes the space of lower semicontinuous functions in $\bar\Omega\times[-T,0]$.

With respect to the time derivative, it is natural for the parabolic topology to consider only the values of $u$ towards the past. In this sense
\begin{align*}
u_{t^-}(x,t) = \lim_{h\to0^+}\frac{u(x,t)-u(x,t-h)}{h}.
\end{align*}

\begin{definition}\label{viscosity}
A function $u \in USC(\bar\Omega\times[-T,0])\cap C(-T,0;L^1(\w))$ ($u \in C(-T,0;L^1(\w)) \cap LSC(\bar\Omega\times[-T,0])$), is said to be a subsolution (supersolution) to $u_t - Iu=f$, and we write $u_t - Iu \leq f$ ($u_t - Iu \geq f$), if every time $(v,B_r(x)\times(t-\t,t]) \in S$ touches $u$ from above (below) at $(x,t)$, i.e.
\begin{itemize}
\item[(i)] $v(x,t)=u(x,t)$,
\item[(ii)] $v(y,s)>u(y,s)$ ($\varphi(y,s)<u(y,s)$) for every $(y,s)\in B_r(x)\times(t-\t,t]\sm\{(x,t)\}$,
\item[(iii)] $v(y,s)=u(y,s)$ for every $(y,s)\in\R^n\sm B_r(x)\times(t-\t,t]$,
\end{itemize}
then $v_{t^-}(x,t) - Iv(x,t) \leq f(x,t)$ ($v_{t^-}(x,t) - Iv(x,t) \geq f(x,t)$).
\end{definition}

An equivalent definition holds if instead of using parabolic second order polynomials as test functions we use test functions $\varphi$ with less regularity around the contact point. This is important when we want to prove the maximum principle by means of and inf and sup convolutions. We omit it here and just assume that maximum principle for viscosity solutions holds. The ideas of the proof of this result can be found in Section 4 and 5 in \cite{C1} for the stationary problem and in the appendix of \cite{S2} for the time depending problem.

One important advantage of the viscosity solutions is that they have enough stability to allow us to keep the equation even when the boundary data changes abruptly in time far away from the domain of the equation. For example, consider the function $u(x,t) = \chi_{E\times\{0\}}$ where $E \ss \R^n\sm B_1$. In the domain $B_1\times(-1,0)$, $u$ satisfies $u_{t^-} + (-\D)^{\s/2}u = 0$ in the classical sense. When $t=0$ the equation is not satisfied any more because $u_{t^-}(x,0)$ is still zero in $B_1$ but $(-\D)^{\s/2}u(x,0)$ becomes strictly positive in $B_1$. If we consider now the same equation in the viscosity sense, $u$ is a solution even when $t=0$. The restriction for the test functions to be in $C(-\t,0;L^1(\w))$ (in the case the contact occurs at $t=0$) implies that such test functions will not be able to see that $u$ has a jump at $t=0$.


\subsection{Previous Results}

The following results can be found in \cite{CD}.


\begin{theorem}[Existence and uniqueness]
Let $\s\in(0,2)$, $\W$ be a smooth domain, $I$ a continuous elliptic operator with respect to $\cL_0$ and $f$ and $g$ bounded, continuous functions. The Dirichlet problem, 
\begin{align*}
 u_t - Iu &= f \text{ in } \W\times(-T,0],\\
 u &= g \text{ in } ((\R^n \sm \W)\times(-T,0]) \cup (\R^n\times\{-T\}),
\end{align*}
has a unique viscosity solution $u$.
\end{theorem}

\begin{theorem}[H\"older regularity]
Let $\s_0 \in(0,2)$ and $\s\in(\s_0,2)$. Let $u\in C(\W\times(-1,0])$ such that it satisfies the following two inequalities in the viscosity sense with $C_0 \geq 0$,
\begin{align*}
u_t - \cM_{\cL_0(\s)}^- u &\geq -C_0 \text{  in $B_1\times(-1,0]$},\\
u_t - \cM_{\cL_0(\s)}^+ u &\leq C_0 \text{  in $B_1\times(-1,0]$}.
\end{align*}
Then there is some $\a \in (0,1)$ and $C>0$, depending only on $n$, $\L$ and $\s_0$, such that for every $(y,s),(x,t) \in B_{1/2}\times(-1/2,0]$
\begin{align*}
\frac{|u(y,s) - u(x,t)|}{(|x-y| + |t-s|^{1/\s})^\a} \leq C\1\|u\|_{L^\8(\bar B_1\times[-1,0])} + \|u\|_{C(-1,0;L^1(\w))} + C_0\2
\end{align*}
\end{theorem}

\begin{theorem}[Regularity for translation invariant operators]\label{C1aold}
Let $\s_0 \in(0,2)$, $\s\in(\s_0,2)$, $I$ be elliptic operator with respect to $\cL_1(\s,\L)$ and translation invariant in space and $f\in C([-1,0])$. Let $u \in C(\bar B_1\times[-1,0])$ be a viscosity solution of the equation,
\begin{align*}
u_t - Iu = f(t) \text{ in $B_1\times(-1,0]$},
\end{align*}
then $u$ is $C^{1,\a}$ in space for some universal $\a \in (0,1)$. More precisely, there is a constant $C>0$, depending only on $n$, $\L$ and $\s_0$, such that for every $(x,t),(y,s) \in B_{1/4}\times(-1,0]$
\begin{align*}
\frac{|u_{x_i}(x,t) - u_{x_i}(y,s)|}{(|x-y| + |t-s|^{1/\s})^\a} &\leq C\1\|u\|_{L^\8(\bar B_1\times[-1,0])} + \|u\|_{C(-1,0;L^1(\w))}\right.\\
&\left. {} + \|f\|_{L^\8((-1,0])}\2, \text{ for $i=1,\ldots,n$}.
\end{align*}
\end{theorem}


\subsubsection{Counterexample for time regularity}

Theorem \ref{C1aold} does not give more regularity in time even if $I$ is translation invariant in time and $f \equiv 0$. Consider the following example for the fractional heat equation $u_t + (-\D)^{\s/2}u = 0$ in $B_1\times(-1,0]$. Let $u$ be the solution of such problem, with initial value $u(\cdot,-1) \equiv 0$. The boundary data outside $B_1$ is equal to $\underbar u$ where 
\begin{align*}
\underbar u(x,t) = \begin{cases}
0 &\text{ if $t < -1/2$},\\
C_1(t+1/2) + \chi_{B_3 \sm B_2}(x) &\text{ if $t \geq -1/2$},
\end{cases}
\end{align*}
and the constant $C_1>0$ is chosen small enough so that $\underbar u$ is a subsolution to the fractional heat equation in $B_1\times(-1,0]$. By the comparison principle, we have that $u \geq \underbar u$ in $B_1\times(-1,0]$. Also $u \equiv 0$ in $B_1\times[-1,-1/2]$ by uniqueness. This show that the time derivative of $u$ have a jump at $t=-1/2$ at every $x\in B_1$.

\section{Boundary regularity}\label{SBR}

The regularity up to the boundary depends on the comparison principle and the existence of barriers. In \cite{CD} it was showed the existence of a barrier with the following properties. From this moment on we will implicitly assume that all the integro-differential inequalities are satisfied in the viscosity sense.

\begin{lemma}[Barrier]\label{barrier}
For $\s\geq\s_0>0$, there exists a non negative function $\psi:\R^n\times(-\8,0]\to\R$ such that:
\begin{alignat*}{2}
\psi &= 0 &&\text{ in $B_1\times\{0\}$},\\
\psi_t-\cM^+_{\cL_0}\psi &\geq 0 &&\text{ in $\R^n\sm B_1\times(-\8,0]$},\\
\psi &\geq 1 &&\text{ in $\R^n\times(-\8,0]\sm(B_2\times[-\k,0])$},
\end{alignat*}
for some $\k$ universal (depending on $\s_0$ but independent of $\s$).
\end{lemma}


\begin{theorem}[Regularity up to the boundary]\label{bddryreg}
 Let $\s\geq\s_0 > 0$. Let $u:\R^n\times[-1,0]\to \R$ a bounded function such that
\begin{alignat*}{2}
u_t - \cM^+_{\cL_0} u &\leq C_0 &&\text{ in $B_1\times(-1,0]$},\\
u_t - \cM^-_{\cL_0}u &\geq-C_0 &&\text{ in $B_1\times(-1,0]$}.
\end{alignat*}
Let $\rho$ be a modulus of continuity such that
\begin{align*}
|u(x,t) - u(y,s)| \leq \rho(|x-y|\vee|t-s|)
\end{align*}
for every $(x,t) \in (\p B_1\times[-1,0]) \cup (B_1 \times\{-1\})$ and $(y,s)\in ((\R^n \sm B_1)\times[-1,0]) \cup (B_1 \times\{-1\})$. Then there is another modulus of continuity $\bar\rho$ so that
\begin{align*}
|u(x,t) - u(y,s)| \leq \bar\rho(|x-y|\vee|t-s|)
\end{align*}
for every $(x,t)\in B_1\times[-1,0]$ and $(y,s)\in \R^n \times[-1,0]$. The modulus of continuity $\bar\rho$ depends only on $\rho$, $\L$, $\s_0$, $n$, $\|u\|_\8$ and $C_0$.
\end{theorem}

In the particular case of H\"older boundary data, the same proof will give us a H\"older modulus of continuity up to the boundary.

\begin{corollary}\label{bdrycoro}
Let $u$ be as in Theorem \ref{bddryreg} and $\rho(|x-y|\vee|t-s|)=C(|x-y|^{\a_1}+|t-s|^{\a_1})$ for some universal constants $C,\a_1>0$, then $u\in C^\a(\bar{B}_1\times[0,1])$ for some universal constant $\a \leq \a_1$.
\end{corollary}

We split the proof in four lemmas, according if we are close to the lateral or bottom boundary.


\begin{lemma}\label{bddry1}
Let $\s\geq\s_0 > 0$. Let $u:\R^n\times[-1,0]\to \R$ a bounded function such that
\begin{alignat*}{2}
u_t - \cM^+_{\cL_0}u &\leq C_0 &&\text{ in $B_1\times(-1,0]$}.
\end{alignat*}
Let $\rho$ be a modulus of continuity such that
\begin{align*}
u(y,s) - u(x,t) \leq \rho(|x-y|\vee (t-s))
\end{align*}
for every $(x,t) \in \p B_1\times(-1,0]$ and $(y,s)\in ((\R^n \sm B_1)\times[-1,t]) \cup (B_1\times\{-1\})$. Then there is another modulus of continuity $\bar\rho$ such that
\begin{align*}
u(y,s) - u(x,t) \leq \bar\rho(|x-y|\vee (t-s))
\end{align*}
for every $(x,t)\in \p B_1\times(-1,0]$ and $(y,s)\in \R^n\times[-1,t]$. The modulus of continuity $\bar\rho$ depends only on $\rho$, $\L$, $\s_0$, $n$, $\|u\|_\8$ and $C_0$.
\end{lemma}


\begin{proof}
We can reduce it to the case when $C_0 = 0$ by adding to $u$ a sufficiently smooth function $p$ that makes the right hand side of the equation for $u+p$ smaller than zero. Take for instance $p(x,t) = -C(4-|x|^2)^+$ for some $C$ large enough independent of $\s$ (but depending on $\s_0$).

Fix $(x,t)\in \p B_1\times(-1,0]$, a radius $r\in(0,1)$, $x_r = (1+r)x$ and consider the following barrier constructed from $\psi$ and $\k$ in Lemma \ref{barrier}.
\begin{align*}
 \b_r(y,s) = u(x,t) + \rho(3r\vee\k r^{\s_0}) + 2\|u\|_\8\psi\1\frac{y-x_r}{r},\frac{s-t}{r^\s}\2.
\end{align*}

As a combination of dilation and translations of $\psi$ we have that
\begin{align*}
 (\b_r)_t - \cM^+_{\cL_0}\b_r \geq 0 \text{ in $(\R^n \sm B_r(x_r)) \times [-\8,t] \ss B_1 \times [-1,t]$}.
\end{align*}
By the comparison principle, we have that
\begin{align}\label{barrier_comparison}
\b_r &\geq u \text{ in $\R^n\times[-1,t]$},\\
\nonumber\text{ given that } \b_r &\geq u \text{ in $((\R^n \sm B_1)\times[-1,t]) \cup (B_1\times\{-1\})$.}
\end{align}
We split $((\R^n \sm B_1)\times[-1,t]) \cup (B_1\times\{-1\})$ in two regions, around $x$ and away from $x$, respectively,
\begin{align*}
&((\R^n \sm B_1)\times[-1,t]) \cup (B_1\times\{-1\}) \ss\\
&(B_{3r}(x)\times[(t-\k r^\s)\vee-1,t]) \sm (B_1\times(-1,0])\\
&{} \cup (\R^n\times[-1,t]) \sm(B_{2r}(x_r)\times[t-\k r^\s,t]).
\end{align*}
On the first region we use the modulus of continuity for $u$,
\begin{align*}
\b_r(y,s) &\geq u(x,t) + \rho(3r\vee\k r^{\s_0}),\\
&\geq u(x,t) + \rho(3r\vee\k r^\s),\\
&\geq u(y,s).
\end{align*}
Over the second region we use that $\psi \geq 1$ in $\R^n\times(-\8,0]\sm(B_2\times[-\k,0])$. By the rescaling, we get that in $(\R^n\times(-\8,t]) \sm(B_{2r}(x_r)\times[t-\k r^\s,t])$,
\begin{align*}
\b_r(y,s) &\geq u(x,t) + 2\|u\|_\8 \geq u(y,s).
\end{align*}
We have then proved \eqref{barrier_comparison}. 

Let $\rho_0$ be a modulus of continuity for $\psi$. We construct $\bar\rho$ in the following way,
\begin{align*}
 \bar\rho(d) = \inf_{r\in(0,1)}\3\rho(3r\vee\k r^{\s_0}) + 2\|u\|_\8\rho_0\1d/r^2\2\4.
\end{align*}
Notice that $\rho_0\1d/r^2\2 \geq \psi\1\frac{y-(1+r)x}{r},\frac{s-t}{r^\s}\2$ for $d = |x-y|\vee(t-s)$. Also $\bar\rho$ is positive, decreasing and concave by being and infimum combination of functions of the same type. It also goes to zero when $d$ goes to zero.
\end{proof}


\begin{lemma}\label{bddry2}
 Let $\s\geq\s_0 > 0$. Let $u:\R^n\times[-1,0]\to \R$ a bounded function such that
\begin{alignat*}{2}
u_t - \cM^+_{\cL_0}u &\leq C_0 &&\text{ in $B_1\times[-1,0]$},\\
u_t - \cM^-_{\cL_0}u &\geq-C_0 &&\text{ in $B_1\times[-1,0]$}.
\end{alignat*}
Let $\rho$ be a modulus of continuity such that
\begin{align*}
|u(x,t) - u(y,s)| \leq \rho(|x-y|\vee|t-s|)
\end{align*}
for every $(x,t) \in \p B_1\times(-1,0]$ and $(y,s)\in \R^n\times[-1,t]$. Then there is another modulus of continuity $\bar\rho$ so that
\begin{align*}
|u(x,t) - u(y,s)| \leq \bar\rho(|x-y|\vee|t-s|)
\end{align*}
for every $(x,t)\in B_1\times(-1,0]$ with $(1-|x|)^{\s_0} \leq (1+t)$ and $(y,s)\in \R^n\times(-1,t]$. The modulus of continuity $\bar\rho$ depends only on $\rho$, $\L$, $\s_0$, $n$, $\|u\|_\8$ and $C_0$.
\end{lemma}


\begin{proof}
Let $x_0 \in \p B_1$ such that $r := |x-x_0| = \dist(x,\p B_1)$ and assume $(y,s) \in B_1\times(-1,t]$. We split the proof in this case into two situations, depending if $(y,s)$ is contained or not in $B_{r/2}(x)\times[t-(r/2)^\s,t]$. Notice that $B_r(x)\times[t-r^\s,t]$ is contained in $B_1\times[-1,0]$ because $(1-|x|)^\s \leq (1-|x|)^{\s_0} \leq (1+t)$.

If $(y,s) \notin B_{r/2}(x)\times[t-(r/2)^\s,t]$ then we have that,
\begin{align*}
r/2 \leq |x-y|\vee(t-s)^{1/\s} \leq |x-y|\vee(t-s)^{1/2}.
\end{align*}
Now we use the modulus of continuity from $(x_0,t)$,
\begin{align*}
|u(x,t)-u(y,s)| &\leq \rho(r) + \rho(|x_0-y|\vee(t-s)),\\
&\leq \rho(2(|x-y|\vee(t-s)^{1/2})) + \rho((|x-y|+r)\vee(t-s)),\\
&\leq \rho(2(|x-y|\vee(t-s)^{1/2})) + \rho(3(|x-y|\vee(t-s)^{1/2})),\\
&\leq 2\rho(3(|x-y|\vee(t-s))^{1/2}).
\end{align*}

If $(y,s) \in B_{r/2}(x)\times[t-(r/2)^\s,t]$ then we use the interior estimates in $B_r(x)\times [t-r^\s,t]$. Rescale $B_r(x)\times[t-r^\s,t]$ to $B_1\times[-1,0]$. So that,
\begin{align*}
v(\xi,\tau) = u(x + r\xi, t + r^\s \tau) - u(x_0,t),
\end{align*}
satisfies,
\begin{alignat*}{2}
v_t - \cM^+_{\cL_0}v &\leq r^\s C_0 &&\text{ in $B_1\times[-1,0]$},\\
v_t - \cM^-_{\cL_0}v &\geq-r^\s C_0 &&\text{ in $B_1\times[-1,0]$}.
\end{alignat*}
Moreover, the modulus of continuity from $(x_0,t)$ allows us to control the norms $\|v\|_{L^\8(B_1\times[-1,0])}$ and $\|v\|_{C(-1,0,L^1(\w))}$.
\begin{align*}
\|v\|_{L^\8(B_1\times[-1,0])} &\leq \rho(2r\vee r^{\s_0}),\\
\|v\|_{C(-1,0,L^1(\w))} &\leq \int \frac{\rho(r(1+|\xi|)\vee r^{\s_0})}{1+|\xi|^{n+\s_0}}d\xi := I_r.
\end{align*}
Notice that $I_r$ goes to zero as $r$ goes to zero by monotone convergence. Here we are assuming that  every modulus of continuity $\r$ is decreasing and bounded.

By the interior estimates applied to $v$ we get that for $(y,s) \in B_{r/2}(x)\times[t-(r/2)^\s,t],$
\begin{align*}
|u(y,s)-u(x,t)| &= |v(\xi,\tau) - v(0,0)|,\\
&\leq C(r^\s C_0 + \rho(2r\vee r^{\s_0}) + I_r)(|\xi|^\a \vee |\tau|^{\a/\s}),\\
&\leq C\frac{m_r}{r^\a}\1|x-y|\vee(t-s)^{1/\s}\2^\a,
\end{align*}
where $m_r = r^\s C_0 + \rho(2r\vee r^{\s_0}) + I_r$. This gives us a modulus of continuity in terms of $d = |x-y|\vee(t-s)^{1/\s}$ if the following quantity goes to zero as $d$ goes to zero,
\begin{align*}
\sup_{r\in[2d,1]}\frac{m_r d^\a}{r^\a}.
\end{align*}
The quotient $d^\a/r^\a$ is bounded by above by $2^{-\a}$. If the supremum above is taken for some sequence of $r$'s going to zero, then $m_r$ goes to zero and also the expression. If the $r$'s remain away from zero then the supremum goes to zero because $d$ goes to zero. Then we have a modulus $\bar \rho$, independent of $\s$, so that
\begin{align*}
 |u(y,s)-u(x,t)| \leq \bar\rho(|x-y|\vee(t-s)^{1/\s}).
\end{align*}
This still depends on $\s$ but we can take care of that by noticing that,
\begin{align*}
 \bar\rho(|x-y|\vee(t-s)^{1/\s}) \leq \bar\rho(|x-y|\vee(t-s)^{1/2}).
\end{align*}

\end{proof}


\begin{lemma}\label{ival1}
 Let $\s\geq\s_0 > 0$. Let $u:\R^n\times[-1,0]\to \R$ a bounded function such that
\begin{alignat*}{2}
u_t - \cM^+_{\cL_0}u &\leq C_0 &&\text{ in $B_1\times[-1,0]$}.
\end{alignat*}
Let $\rho$ be a modulus of continuity for $u$ such that
\begin{align*}
u(y,s) - u(x,-1) \leq \rho(|x-y|\vee (s+1))
\end{align*}
for every $(x,-1) \in B_1\times\{-1\}$ and $(y,s)\in ((\R^n \sm B_1)\times[-1,0]) \cup (B_1\times\{-1\})$. Then there is another modulus of continuity $\bar\rho$ so that
\begin{align*}
u(y,s) - u(x,-1) \leq \bar\rho(|x-y|\vee (s+1))
\end{align*}
for every $(x,-1) \in B_1\times\{-1\}$ and $(y,s)\in \R^n\times[-1,0]$. The modulus of continuity $\bar\rho$ depends only on $\rho$, $\L$, $\s_0$, $n$, $\|u\|_\8$ and $C_0$.
\end{lemma}


\begin{proof}
Assume as in the proof of Lemma \ref{bddry1} that $C_0 = 0$. Fix $(x,-1)\in B_1\times\{-1\}$. Let $b :\R^n \to [0,1]$ a smooth bump function such that $\supp(1-b) = B_1$ and $b(0) = 0$. Then $\psi_t - \cM^+_{\cL_0}\psi \geq 0$ for,
\begin{align*}
 \psi(y,s) = b(y) + \|\cM^+_{\cL_0}b\|_\8 s.
\end{align*}

Consider the following barrier for a given radius $r\in(0,1)$,
\begin{align*}
\b_r(y,s) = u(x,-1) + \rho(r) + 2\|u\|_\8 \3 \psi \1\frac{y-x}{r},\frac{s+1}{r^\s}\2 + \frac{s+1}{r}\4.
\end{align*}
$\b_r$ it is constructed such that
\begin{align*}
(\b_r)_t - \cM^+_{\cL_0}\b_r \geq 0 \text{ in $\R^{n+1}$}.
\end{align*}
Let's see that $\b_r \geq u$ in $((\R^n\sm B_1)\times[-1,0]) \cup (B_1\times\{-1\})$ and therefore also in $\R^n\times[-1,0]$ by the maximum principle. If $|x-y|\vee(s+1) \leq r$ then $\b_r(y,s) \geq u(x,-1) + \rho(r) \geq u(y,s)$. On the other side, if $|x-y|\vee(s+1) \geq r$ then either $|x-y| \geq r$ and then $b((y-x)/r) = 1$ or $(s+1) \geq r$ and then $(s+1)/r \geq 1$, in any case
\begin{align*}
 \b_r(y,s) \geq -\|u\|_\8 + 2\|u\|_\8 \3 \psi \1\frac{y-x}{r},\frac{1+s}{r^\s}\2 + \frac{s+1}{r}\4 \geq u(y,s).
\end{align*}

Having that $\b_r \geq u$ we get that the following modulus of continuity satisfies the conclusion of the lemma
\begin{align*}
 \bar\rho(d) = \inf_{r\in(0,1)} \3\rho(r) + 2\|u\|_\8\1\rho_0 \1d/r^2\2+d/r\2\4
\end{align*}
where $\rho_0$ is a modulus of continuity for $\psi(y,s)$ and $d=|x-y|\vee(s+1)$.
\end{proof}


\begin{lemma}\label{ival2}
 Let $\s\geq\s_0 > 0$. Let $u:\R^n\times[-1,0]\to \R$ a bounded function such that
\begin{alignat*}{2}
u_t - \cM^+_{\cL_0}u &\leq C_0 &&\text{ in $B_1\times[-1,0]$}.
\end{alignat*}
Let $\rho$ be a modulus of continuity for $u$ such that
\begin{align*}
|u(y,s) - u(x,-1)| \leq \rho(|x-y|\vee (s+1))
\end{align*}
for every $(x,-1) \in B_1\times\{-1\}$ and $(y,s)\in \R^n\times[-1,0]$. Then there is another modulus of continuity $\bar\rho$ so that
\begin{align*}
|u(y,s) - u(x,t)| \leq \bar\rho(|x-y|\vee (t-s))
\end{align*}
for every $(x,t)\in B_1\times[-1,0]$ with $(1-|x|)^{\s_0} \geq (1+t)$ and $(y,s)\in \R^n\times[-1,t]$. The modulus of continuity $\bar\rho$ depends only on $\rho$, $\L$, $\s_0$, $n$, $\|u\|_\8$ and $C_0$.
\end{lemma}


\begin{proof}
Fix $(x,t)\in B_1\times[-1,0]$ with $(1-|x|)^{\s_0} \geq (1+t)$. Let $r^\s=1+t$. As in the proof of Lemma \ref{bddry2} we split it into two situations depending if $(y,s)$ is contained or not in $B_{r/2}(x)\times[t-(r/2)^\s,t]$. In the case when $(y,s) \notin B_{r/2}(x)\times[t-(r/2)^\s,t]$ we estimate in the following way. Recall first that $r/2 \leq |x-y|\vee(t-s)^{1/2}$. Then,
\begin{align*}
|u(y,s)-u(x,t)| &\leq \rho(|x-y|\vee (s+1)) + \rho(r^\s),\\
&\leq \rho(|x-y|\vee r^\s) + \rho(16(|x-y|^{\s_0}\vee (t-s)^{\s_0/2})),\\
&\leq 2\rho(16(|x-y|^{\s_0}\vee (t-s)^{\s_0/2})).
\end{align*}

In the case when $(y,s) \in B_{r/2}(x)\times[t-(r/2)^\s,t]$ we use the interior estimates as in the proof of Lemma \ref{bddry2}. We rescale $B_r(x)\times[t-r^\s,t]$ to $B_1\times[-1,0]$ to obtain
\begin{align*}
|u(y,s)-u(x,t)| &\leq C(\rho(r\vee r^{\s_0}) + r^\s C_0 + I_r)\1\frac{|x-y|\vee(t-s)^{1/\s}}{r}\2^\a,\\
&\leq C\frac{m_r}{r^\a}\1|x-y|\vee(t-s)^{1/\s}\2^\a,
\end{align*}
where $I_r$ is the term that comes from the $C(-1,0,L^1(\w))$ norm and $m_r = \rho(r\vee r^{\s_0}) + r^\s C_0 + I_r$ goes to zero as $r$ goes to zero uniformly in $\s \in[\s_0,2)$. As before this gives a modulus of continuity independent of $\s$.

\end{proof}


\begin{proof}[Proof of Theorem (\ref{bddryreg})]
By applying the four previous lemmas we get a modulus of continuity for $(x,t) \in B_1\times[-1,0]$ and $(y,s) \in \R^n \times[-1,0]$ if $s\leq t$. Now if $s\geq t$ then we only have to consider $y \in \R^n \sm B_1$, otherwise we just have to interchange the points and apply the lemmas again. In this case let $x_0 \in \p B_1$ with $r = \dist(x,\p B_1) = |x-x_0|$. Then
\begin{align*}
|u(x,t)-u(y,s)| &\leq \bar\rho(r) + \rho(|x_0-y|\vee(s-t)),\\
&\leq 2\bar\rho(|x-y|\vee(s-t)).
\end{align*}
This tell us how we need to modify the modulus of continuity to conclude with the theorem.
\end{proof}


\section{Further regularity in time}\label{SFRT}

In this section we impose some regularity to the boundary data in order to get better regularity estimates in the time variable. 



\begin{theorem}\label{regularityint}
Let $u:\R^n\times[-1,0]\to \R$ be a viscosity solution of
\begin{align*}
u_t-Iu=0, \ \ \hbox{in}\ B_1\times(-1,0],
\end{align*}
where $I$ is elliptic with respect to $\cL_0$. Assume that $u$ is Lipschitz in $\R^n\sm B_1\times(-1,0]$ and that for $u_0(x)=u(x,-1)$, $|\cM^\pm_{\cL_0}u_0| < C_0$. Then for $(x,t), (y,s)\in B_{1/2}\times(-1/2,0]$ we have 
\begin{align*}
\frac{|u_t(x,t)-u_t(y,s)|}{(|x-y|+|t-s|^{1/\s})^\a}\leq& C\1\|u\|_{L^\infty(\bar{B}_1\times[-1,0])}+C_0+\|u_t\|_{L^\infty(\R^n\sm B_1\times[-1,0])})\2
\end{align*}
Here $\a>0$ and $C$ only depend on $C_0$ and the ellipticity constants.
\end{theorem}



As pointed out before, at the end of Section 2, $C^{1,\a}$ regularity in time is not always available. By the non locality, any discontinuity it is immediately seen by the solution. In order to bypass such obstacle we decided to ask enough regularity to the solution and use the comparison principle to propagate such regularity to the interior of the domain. After doing this we can improve the regularity a little bit more by the interior estimates.

\begin{proof}[Proof of Theorem \ref{regularityint}]
Define $M$ by
\begin{align*}
M=\max\{C_0, \|u_t\|_{L^{\infty}(\R^n\sm B_1\times(-1,0])}\},
\end{align*}
and construct the barrier
\begin{align*}
\varphi(x,t)=u_0(x)+M(1+t).
\end{align*}
Note now that $\varphi$ satisfies, 
\begin{align*}
\varphi_t(x,t)-\cM^+_{\cL_0}\varphi(x,t)=M-\cM^+_{\cL_0}u_0\geq 0.
\end{align*}
Moreover $\varphi\geq u$ in $\R^n\sm B_1\times[-1,0]\cup \R^n\times\{-1\}$. Hence by the comparison principle we conclude $\varphi\geq u$ in $\R^n\times[-1,0]$. In particular this implies
\begin{align*}
u(x,t)-u(x,-1)\leq M(t+1).
\end{align*}
To get the other side of the bound just consider $\varphi(x,t)=u_0(x)-M(1+t)$ instead. This gives us Lipschitz regularity at $t=-1$.

To get Lipschitz regularity for $t>-1$ we define
\begin{align*}
w(x,t)=u(x,t+h)-u(x,t),
\end{align*}
for $t$ in $(-1,-h]$ and $x$ in $\R^n$. From the previous computations $w(x,-1)\leq Mh$. Also, since $u$ is Lipschitz in $\R^n\sm B_1\times(-1,0]$, we have the estimate $w(x,t)\leq Mh$ in $(\R^n\sm B_1)\times(-1,-h]$. Note that $w$ satisfies  $w_t-\cM^+w\leq 0$ (see Theorem 2.1 in \cite{CD}). We conclude by the maximum principle that
\begin{align*}
\max\limits_{\R^n\times[-1,-h]} w&\leq \max\limits_{(\R^n\sm B_1)\times[-1,h]}\\
&\leq Mh,
\end{align*}
which gives us one side of the inequality. To get the other inequality we consider $w=u(x,t)-u(x,t+h)$ instead.

Finally for $h>0$ define
\begin{align*}
w(x,t) = \frac{u(x,t+h)-u(x,t)}{h},
\end{align*}
and thanks to the previous computations we know that $w$ is bounded in $\R^n \times (-1,-h]$, uniformly in $h$, and therefore also belongs to $L^1(\w)$ uniformly in time. Moreover $w$ satisfies the following equations by being the difference of two solutions of a translation invariant equation:
\begin{alignat*}{2}
w_t - \cM^+_{\cL_0} w &\leq 0 &&\text{ in $B_1\times(-1,-h]$},\\
w_t - \cM^-_{\cL_0}w &\geq 0 &&\text{ in $B_1\times(-1,-h]$}.
\end{alignat*}
We can use now the interior H\"older regularity from \cite{CD} to conclude that $w$ is $C^\a(B_{1/2}\times(-1/2,-h])$ uniformly in $h$. Passing to the limit $h\to 0$ we conclude the desired result.
\end{proof}

\begin{remark}
The proof in Theorem \ref{regularityint} works also in the case when the boundary data has a general modulus of continuity $\rho(\tau)$ in time. In the interior the solution will have a better modulus of continuity. It would be of the form $\tau^\a\rho(\tau)$ if $\limsup_{\t\to0}\tau^{\a-1}\rho(\tau) > 0$ or it would imply a modulus of continuity for $u_t$ of the form $\tau^{\a-1}\rho(\tau)$ if $\limsup_{\t\to0}\tau^{\a-1}\rho(\tau) = 0$. See Lemma 5.6 in \cite{CC}.
\end{remark}



\section{Stability and Compactness}\label{SSC}

\begin{definition}[Weak convergence for non local operators]
We say that a sequence $\{I_k\}$ of continuous operators in $\W\times(-T,0]$ converges weakly to an operator $I$ in $\W\times(-T,0]$ if for every test function $(v,B_\r(x)\times(t-\t,t])\in S$ we have that $I_kv \to Iv$ uniformly in $B_{\r/2}(x)\times(t-\t/2,t]$.
\end{definition}

\begin{lemma}
If a sequence $\{I_k\}$ of continuous operators in $\W\times(-T,0]$ converges weakly to an operator $I$ in $\W\times(-T,0]$ then $I$ is also continuous in $\W\times(-T,0]$.
\end{lemma}

\begin{proof}
Let $(v,B_\r(x)\times(t-\t,t]) \in S$, for $(x_i,t_i)\to(x_0,t_0^-) \in B_\r(x)\times(t-\t,t]$ we can assume that $(x_i,t_i) \in B_{r/2}(x_0)\times(t_0-t_1/2,t_0]$ with $B_r(x_0)\times(t_0-t_1,t_0] \ss B_\r(x)\times(t-\t,t]$. We estimate by the triangular inequality,
\begin{align*}
|Iv(x_i,t_i) - Iv(x_0,t_0)| &\leq |Iv(x_i,t_i) - I_kv(x_i,t_i)| + |I_kv(x_i,t_i) - I_kv(x_0,t_0)|\\
&{} + |I_kv(x_0,t_0) - Iv(x_0,t_0)|.
\end{align*}
The first and third term can be made arbitrarily small by the uniform convergence of $I_kv$ to $Iv$ in $B_{r/2}(x_0) \times (t_0 - t_1/2,t_0]$. The second term can also be made arbitrarily small once we fix $k$ by the continuity of $I_k$.
\end{proof}

\begin{remark}
In the proof of Theorem \ref{compactness} and \ref{stacomp} we will need additionally that the Dirichlet problem given by $I$ has unique solutions. Notice that this will actually be inherited from the fact that the limit operator $I$ is also elliptic with respect to $\cL_0$ is the same holds for the sequence $\{I_k\}$.
\end{remark}

\subsection{Stability}

The definition of $\G$-convergence has to be done with respect to the parabolic topology.

\begin{definition}
We say that a sequence of lower semicontinuous functions (in the parabolic topology) $u_k$ $\G$-converge to $u$ in a set $\W\times(-T,0]$ if the two following conditions hold
\begin{enumerate}
\item[(i)] For every sequence $(x_k,t_k)\to(x,t^-)$ in $\W\times(-T,0]$,
\begin{align*}\liminf_{\k\to\infty}u_k(x_k,t_k)\geq u(x,t).
\end{align*}
\item[(ii)] For every $(x,t)\in\W\times(-T,0]$, there is a sequence $(x_k,t_k)\to(x,t^-)$ in $\W\times(-T,0]$ such that
\begin{align*}
\limsup\limits_{k\to\infty}u_k(x_k,t_k)=u(x,t).
\end{align*}  
\end{enumerate}
\end{definition}

\begin{theorem}[Stability]\label{stability}
Let $I_k$ be a sequence of uniformly elliptic operators with respect to $\cL_0$. Let $\{u_k\}_{k\geq 1} \ss LSC(\bar\W\times(-T,0])\cap C(-T,0;L^1(\w))$ such that
\begin{enumerate}
\item[(i)] $(u_k)_t - I_ku_k \geq f_k$ in the viscosity sense in $\W\times(-T,0]$,
\item[(ii)] $u_k \to u$ in the $\G$ sense in $\W\times(-T,0]$ and in $C(-T,0;L^1(\w))$,
\item[(iii)] $f_k \to f$ locally uniformly in $\W\times(-T,0]$,
\item[(iv)] $I_k \to I$ weakly in $\W\times(-T,0]$ (with respect to $\w$),
\item[(v)] $|u_k(x,t)| \leq C$ for every $(x,t) \in \W\times(-T,0]$.
\end{enumerate}
Then also $u_t - Iu \geq f$ in the viscosity sense in $\W\times(-T,0]$.
\end{theorem}

\begin{proof}
Let $p$ a parabolic second order polynomial touching $u$ from below at a point $(x,t)$ in a neighborhood $B_r(x)\times(t-r,t]$.

Since $u_k$ $\G$-converges to $u$ in $\W\times(-T,0]$, for large $k$, we can find $(x_k,t_k)\to(x,t^-)$ and $d_k\to0$ such that $p+d_k$ touches $u_k$ at $(x_k,t_k) \in B_r(x)\times(t-r,t]$. If we let
\begin{align*}
v_k &= \begin{cases}
p+d_k &\text{ in } B_r(x)\times(t-r,t],\\
u_k &\text{ in } \R^{n+1} \sm (B_r(x)\times(t-r,t]),
\end{cases},\\
v &= \begin{cases}
p &\text{ in } B_r(x)\times(t-r,t],\\
u &\text{ in } \R^{n+1} \sm (B_r(x)\times(t-r,t]).
\end{cases}
\end{align*}
For $(y,s) \in B_{r/2}(x)\times(t-r/2,t]$,
\begin{align*}
|I_kv_k(y,s) - Iv(x,t)| &\leq |I_kv_k(y,s) - I_kv(y,s)| + |I_kv(y,s) - Iv(y,s)|\\
&{} + |Iv(y,s) - Iv(x,t)|.
\end{align*}

The last term goes to zero by the continuity of $I$. The second term goes to zero too and it does it independently of $(y,s) \in B_{r/2}(x)\times(t-r/2,t]$ by the weak convergence of $I_k$ to $I$. We use the ellipticity to look at the first term,
\begin{align*}
|I_kv_k(y,s) - I_kv(y,s)| &\leq C\int_{\R^n \sm B_{r/2}}\frac{|\d(v_k-v,y,s;z)|}{|z|^{n+\s}}dz,\\
&\leq C(r)(d_k + \|u_k(\cdot,s)-u(\cdot,s)\|_{L^1(\w)}).
\end{align*}
So this term also goes to zero because $u_k \to u$ in $C(-T,0;L^1(\w))$.

Looking at the equation we have
\begin{align*}
(v_k)_{t^-}(x_k,t_k) - I_kv_k(x_k,t_k) \geq f_k(x_k,t_k).
\end{align*}
Sending $k\to\8$ makes all the term go to the respective ones. $(v_k)_{t^-}(x_k,t_k) = v_{t^-}(x,t)$ because of the way they are defined for $(x_k,t_k) \in B_r(x)\times(t-r,t]$, $I_kv_k(x_k,t_k) \to Iv(x,t)$ was just shown and $f_k(x_k,t_k)\to f(x,t)$ is one of the hypothesis. This completes the proof.
\end{proof}


\subsection{Compactness}

\begin{lemma}\label{comp1}
Let $I_k$ be a sequence of continuous, translation invariant elliptic operators with respect to $\cL_0$. Let $(v,B_r\times(-\t,0]) \in S$ be a fixed test function. There is a subsequence $I_{k_j}$ such that $f_{k_j} := v_{t^-} - I_{k_j}v$ converges uniformly in $B_{r/2}\times(-\t/2,0]$.
\end{lemma}

\begin{proof}
We prove that there is a uniform modulus of continuity with respect to the parabolic topology for $f_k$ in $B_{r/2}\times(-\t/2,0]$ and use Arzela-Ascoli to get the uniform convergence up to a subsequence. The functions $f_k$ are uniformly bounded because $v_{t^-}$ is constant in $B_{r/2}\times(-\t,0]$ and from $|I_kv(x,t)| \leq |I0(x,t)| + |\cM^+_{\cL_0}v(x,t)|$, we use that the right hand side is a continuous function varying over a compact set.

Looking at the difference for $(x,t), (y,s) \in B_{r/2}\times(-\t,0]$ with $|x-y|\leq r/8$ and $t\leq s$.
\begin{align*}
f_k(x,t) - f_k(y,s) &= (v_t - I_kv)(x,t) - (v_t - I_kv)(y,s),\\
&\leq (I_kv - I_k\t_{(y-x,s-t)}v)(x,t),\\
&\leq \cM^+_{\cL_0}(v-\t_{(y-x,s-t)}v)(x,t),\\
&\leq C\int_{\R^n \sm B_{r/4}}\frac{|\d(v-\t_{(y-x,s-t)}v,x,t;z)|}{|z|^{n+\s}}dz,\\
&\leq C(r)(|v(x,t)-v(y,s)| + \|v(\cdot+x,t)-v(\cdot+y,s)\|_{L^1(\w)})\\
&\leq C(r)(|v(x,t)-v(y,s)| + \|v(\cdot,t)-v(\cdot,s)\|_{L^1(\w)}).
\end{align*}
Which gives us the desired modulus of continuity for $(x,t), (y,s) \in B_{\r/2}\times(-\t/2,0]$ with $|x-y|\leq \r/8$ and $s\to t^-$.
\end{proof}

\begin{theorem}[Compactness]\label{compactness}
Let $I_k$ be a sequence of operators, elliptic with respect to $\cL_0$ in $\W\times(-T,0]$. Then there is a subsequence $I_{k_j}$ that converges weakly in $\W\times(-T,0]$ to another operator $I$ also elliptic with respect to $\cL_0$ in $\W\times(-T,0]$.
\end{theorem}

\begin{proof}
First note that $C(-T,0;L^1(\w))$ is a separable space and there exists a dense sequence $\{u_i\}_{i\geq 1} \ss C(-T,0;L^1(\w))$. Let also $\{p_j\}_{j\geq1}$ be a numeration of all possible quadratic parabolic polynomials with rational coefficients. By a diagonal argument we can construct a sequence $\{(v_i,B_{r_i}(x_i)\times(t_i-\t_i,t_i])\}_{i\geq1} \ss S$ such that it covers all possible functions of the form
\begin{align*}
v = \begin{cases}
p_j &\text{ in $B_{r}(x)\times(t-\t,t]$},\\
u_i &\text{ in $B_{r}(x)\times(t-\t,t]$},
\end{cases}
\end{align*}
where $r$, $x$, $\t$ and $t$ vary over all possible rational numbers.

For each $v_i$ we can apply Lemma \ref{comp1} and, after a diagonal argument, extract a subsequence still denoted by $I_k$ such that for every $i$, $\{I_kv_i\}_k$ converges uniformly in $B_{r_i/2}(x_i)\times(t_i-\t_i/2,t_i]$ to some continuous function in $B_{r_i/2}(x_i)\times(t_i-\t_i/2,t_i]$, which we will denote $Iv_i(x,t)$.

We need to show that for any other $(v,B_r(x)\times(t-\t,t]) \in S$, $I_kv$ also converges uniformly in $B_{r/2}(x)\times(t-\t/2,t]$ to some continuous function in $B_{r/2}(x)\times(t-\t/2,t]$, that we will denote $Iv(x,t)$. Finally, we will need to show that $I$, which it is a priori defined on $S$, can be extended to be an elliptic operator with respect to $\cL$ in $\W\times(-T,0]$.

Let $(v,B_r(x)\times(t-\t,t]) \in S$. For $\e>0$ fixed we want to show that there exists some $k_0$ such that for every $k,j \geq k_0$ and every $(y,s) \in B_{r/2}(x)\times(t-\t/2,t]$
\begin{align*}
|I_kv(y,s) - I_jv(y,s)| < \e.
\end{align*}
We proceed by the triangular inequality using the sequence $\{(v_i,B_{r_i}(x_i)\times(t_i-\t_i,t_i])\}_{i\geq1}$ of the test functions defined above. Let $(v_i,B_{r_i}(x_i)\times(t_i-\t_i,t_i])$ such that $B_r(x)\times(t-\t,t] \ss B_{r_i}(x_i)\times(t_i-\t_i,t_i]$. We use the ellipticity to estimate, for any $(y,s) \in B_{r/2}(x)\times(t-\t/2,t]$,
\begin{align*}
|I_kv(y,s) - I_kv_i(y,s)| \leq \sup_{L \in \cL} |L(v-v_i)(y,s)|.
\end{align*}
If $L \in \cL$ has associated the kernel $K$ then,
\begin{align*}
|L(v-v_i)(y,s)| &\leq \int |\d(v-v_i,y,s;z)|K(z)dz,\\
&= \int_{B_r(x)} + \int_{\R^n \sm B_r(x)}.
\end{align*}
The first integral only sees the polynomials parts of $v$ and $v_i$ and can be made smaller that $\e/6$, uniformly in $L\in\cL$, if the coefficients of $v_i$ are sufficiently close to the coefficients of $v$, because of the uniform integrability condition of the kernels. The second integral can also be made smaller than $\e/6$ if $v_i$ is sufficiently close to $v$ in the $C(-T,0;L^1(\w))$ norm. We conclude that there exists some $v_i$ such that for every $k$ and any $(y,s) \in B_{r/2}(x)\times(t-\t/2,t]$,
\begin{align*}
|I_kv(y,s) - I_kv_i(y,s)| \leq \frac{\e}{3}.
\end{align*}
Now we choose $k_0$ such that for every $k,j \geq k_0$ and every $(y,s) \in B_{r_i/2}(x_i)\times(t_i-\t_i/2,t_i]$
\begin{align*}
|I_kv_i(y,s) - I_kv_i(y,s)| \leq \frac{\e}{3}.
\end{align*}
Then by the triangular inequality,
\begin{align*}
|I_kv(y,s) - I_jv(y,s)| &\leq |I_kv(y,s) - I_kv_i(y,s)| + |I_kv_i(y,s) - I_jv_i(y,s)|\\
&{} + |I_jv_i(y,s) - I_jv(y,s)|,\\
&\leq \e.
\end{align*}
Therefore $Iv(y,s)$ is a well defined continuous function in $B_{r/2}(x)\times(t-\t/2,t]$.

Moreover, the ellipticity gets also inherited for functions in $S$ by passing to the limit in the following expression,
\begin{align*}
\cM^-_\cL(v_1-v_2) \leq I_kv_1 - I_kv_2 \leq \cM^+_\cL(v_1-v_2). 
\end{align*}

For a function $u$ and a fixed point $(x,t)$, such that $u(\cdot,t) \in C^{1,1}(x) \cap L^1(\w)$, we have quadratic polynomials $p$ and $q$, independent of time, such that $p$ and $q$ touch $u$ by above  and below respectively at $x$ in $B_r(x)$. We would like to define $Iu(x,t)$ to be
\begin{align*}
Iu(x,t) = \limsup_{\r\to 0} Iv_{p,\r}(x,t)
\end{align*}
where $v_{p,\r} \in S$ is independent of time, such that
\begin{align*}
v_{p,\r} = \begin{cases}
p &\text{ in $B_\r(x)$},\\
u(\cdot,t) &\text{ in $\R^n \sm B_\r(x)$}.
\end{cases}
\end{align*}
All we need to check is that $Iu(x,t)$ is finite and does not depend on $p$.

$Iv_{p,\r}(x,t)$ is bounded by below by $Iv_{q,r}(x,t) > -\8$, because of the monotonicity of $I$ in $S$ (which follows from the ellipticity). In a similar way $Iv_{p,\r}(x,t)$ is also bounded by above by $Iv_{p,r}(x,t) < \8$. Therefore $\limsup_{\r\to 0} Iv_{p,\r}(x,t)$ is a finite number.

To check that the definition of $Iu(x,t)$ is independent of $p$ we need to see that for any other polynomial $P$, independent of time, such that $P$ also touches $u$ by above at $x$ in some neighborhood,
\begin{align*}
\limsup_{\r\to 0} Iv_{p,\r}(x,t) = \limsup_{\r\to 0} Iv_{P,\r}(x,t).
\end{align*}
By the ellipticity of $I$ in $S$,
\begin{align*}
|Iv_{p,\r}(x,t) - Iv_{P,\r}(x,t)| \leq \sup_{L\in\cL} |L(v_{p,\r}-v_{P,\r})(x,t)|.
\end{align*}
For a given $L\in\cL$ with kernel $K$,
\begin{align*}
|L(v_{p,\r}-v_{P,\r})(x,t)| \leq \int_{B_\r} |\d(p-P,x,t;y)|K(z)dz,
\end{align*}
which can be arbitrarily small, uniformly in $L$, because of the uniform integrability condition for the kernels. Therefore the definition of $Iu(x,t)$ is consistent and once again the ellipticity for $I$ gets inherited from the ellipticity for $I_k$.
\end{proof}


\begin{definition}[Norm of a non local operator]
Given a nonlocal operator $I$ with respect to $\w$ and $\W\times(-T,0]$ we define $\|I(t)\|$, for $t \in (-T,0]$ as
\begin{align*}
\|I(t)\| := &\sup\{|I(u,x,t)|/(1+M) : x \in \W, u(\cdot,t) \in L^1(\w) \cap C^{1,1}(x)\\
&\|u(\cdot,t)\|_{L^1(\w)} \leq M,\\
&|u(y,t) - u(x,t) - (y-x)\cdot u_x(x,t)| \leq M|y-x|^2 \text{ for } y\in B_1(x)\}.
\end{align*}
We also write $\|I\| := \sup_{t\in(-T,0]}\|I(t)\|$.
\end{definition}


\begin{theorem}[Stability by compactness]\label{stacomp}
For some $\s \geq \s_0 > 1$ and $\a < \s_0-1$ we consider operators $I_0$, $I_1$ and $I_2$ elliptic respect to $\cL_0(\s,\L)$ with $\w = 1/(1+|y|^{n+\s})$. Then, given $M > 0$, a modulus of continuity $\r$ and $\e > 0$, there is an $\eta > 0$ sufficiently small and $R>0$ sufficiently large so that if $I_0$, $I_1$ and $I_2$ satisfy
\begin{align*}
\|I_1 - I_0\| &\leq \eta,\\
\|I_2 - I_0\| &\leq \eta,\\
\end{align*}
the functions $u,v \in C(B_1\times[-1,0])\cap C(-1,0;L^1(\w))$ satisfy in the viscosity sense,
\begin{alignat*}{2}
v_t - I_0(v,x) &= 0 &&\text{ in $B_1\times[-1,0]$},\\
u_t - I_1(u,x) &\geq -\eta &&\text{ in $B_1\times[-1,0]$},\\
u_t - I_2(u,x) &\leq \eta &&\text{ in $B_1\times[-1,0]$},\\
u&=v &&\text{ in $((\R^n \sm B_1)\times[-1,0])\cup(B_1\times\{-1\})$},
\end{alignat*}
and for every $(x,t)\in (B_R\sm B_1)\times[-1,0] \cup (B_R\times\{-1\})$ and $y\in ((\R^n\sm B_1)\times[-1,0])\cup \R^n\times\{-1\}$.
\begin{align*}
|u(x,t) - u(y,s)| &\leq \rho(|x - y|\vee|t-s|),\\
|u(y,s)| &\leq M\max(|y|^{1+\a},1),
\end{align*}
then $|u-v|<\e$ in $B_1\times[-1,0]$.
\end{theorem}

\begin{proof}
Assume that for some $\e>0$ the result does not hold. Then there is a sequence $R_k$, $I^k_0$, $I^k_2$, $I^k_2$, $\eta_k$, $u_k$, $v_k$, $I_k$, $f_k$ such that $R_k \to \8$, $\eta_k \to 0$, all the assumptions of the lemma are valid but $\sup_{B_1\times[-1,0]}|u_k - v_k| \geq \e$.

By Lemma \ref{compactness} we can assume that $I^k_0 \to I$ weakly. Because $I^k_1$, $I^k_2$ get closer and closer in norm to $I^k_0$ we also get that they converge weakly to $I$. 

By the boundary regularity we get that each $u_k$ and $v_k$ is not only equicontinuous in $((\R^n \sm B_1)\times[-1,0])\cup(B_1\times\{-1\})$ but also in $B_{R_k} \times[-1,0]$. By Arzela-Ascoli we can extract a subsequence (denoted the same) that converges uniformly on compact sets to functions $u$ and $v$ respectively. By dominated convergence theorem, $u_k$ and $v_k$ also converge in $C(-1,0;L^1(\w))$. Indeed,
\begin{align*}
\sup_{t\in(-1,0]} \|u_k(\cdot,t) - u(\cdot,t)\|_{L^1(\w)} &= \int \frac{\sup_{t\in(-1,0]}|u_k(y,t) - u(y,t)|}{1+|y|^{n+\s}}dy,
\end{align*}
and the integrand is dominated by
\begin{align*}
\frac{\sup_{t\in(-1,0]}|u_k(y,t) - u(y,t)|}{1+|y|^{n+\s}} \leq \frac{2M\max(|y|^{1+\a},1)}{1+|y|^{n+\s}} \in L^1.
\end{align*}

Then by the Stability Theorem \ref{stability}, $u$ and $v$ solve the same equation $w_t - Iw = 0$ with the same boundary data therefore by maximum principle $u = v$ and this contradicts that $\sup_{B_1\times[-1,0]}|u_k - v_k| \geq \e$.
\end{proof}


\section{$C^{1,\a}$ Estimates}\label{C1a section}

In this section we prove parabolic $C^{1,\a}$ estimates for equations that are close to be translation invariant. In particular, by applying this at very small scales, the regularity also applies for equations where the operator is continuous.

The strategy consists in proving an improvement of flatness for the graph of the solution $u$ and iterate it at smaller scales. This requires that the norm operator $I$ remains bounded uniformly on the small scales. For this we introduce the norm at scale $\s$ of $I$,
\begin{align*}
\|I(t)\|_\s=\sup\limits_{\b<1}\left\|I_\b(t)\right\|,
\end{align*}
where $I_\b u(x,t) = \b^\s (Iu(\b^{-1}\cdot, \b^{-\s}\cdot))(\b x,\b^\s t)$. Notice that by rescaling $I$, also the domain in $\R^n\times\R$ where $I$ is defined gets rescaled and then the norm $\|I_\b\|$ has to be computed with respect to such domain.
 

\begin{theorem}\label{C1a theorem}
Assume $\s>\s_0>1$. Let $I^{(0)}$ be a fixed translation invariant nonlocal operator in a class $\cL\ss\cL_0$ with scale $\s$, such that the equation $v_t-I^{(0)}v=0$ has interior $C^{1,\bar \a}$ estimates in space and time given that the boundary and initial data is regular enough (See Theorems \ref{C1aold} and \ref{regularityint}). Let $I^{(1)}$ and $I^{(2)}$ be two nonlocal operators, elliptic with respect $\cL_0$, translation invariant only in time and assume that
\begin{align*}
\left\|I^{(0)}-I^{(j)}\right\|_\s<\eta,\ \ j=1,2
\end{align*}
for some small $\eta$.

Let $u$ be a bounded function that solves 
\begin{align*}
u_t-I^{(1)}u\leq f_1(x,t) \text{ in } B_1\times(-1,0]\\
u_t-I^{(2)}u\geq f_2(x,t) \text{ in } B_1\times(-1,0]
\end{align*}
for a couple of bounded functions $f_1, f_2$. Furthermore, assume that $u$ is Lipschitz in time in $\R^n\sm B_1\times(-1,0]$ and that for $u_0(x)=u(x,-1)$, $|\cM^\pm_{\cL_0}u_0|$ is bounded in $B_1$.

Then for $\a<\min(\bar \a,\s_0-1)$ we have that $u$ is $C^{1,\a}(B_{1/2}\times(-1/2,0])$ in all of its variables. Moreover we have for $(x,t), (y,s) \in B_{1/2}\times(-1/2,0]$,
\begin{align*}
\frac{|u_{x_i}(x,t)-u_{x_i}(y,s)|}{(|x-y|+|t-s|^{1/\s})^\a} &\leq C\1\|u\|_{L^\8(\R^n\times[-1,0])} +\|u_t\|_{L^\8(\R^n\sm B_1\times[-1,0])} \right.\\
&\left. {} + \|\cM^\pm_{\cL_0}u_0\|_{L^\8(B_1)} + \|f_1\|_{L^\8((-1,0])} + \|f_2\|_{L^\8((-1,0])}\2
\end{align*}
for $i=1,\ldots,n+1$, where we use the convention $x_{n+1}=t$.
\end{theorem}


\begin{proof}
By scaling the problem, we can assume that 
\begin{align}
\label{cond00} &\|f_{j}\|_\8 < \eta/2, \ \ j=1,2\\
\label{cond0} &\|u\|_{L^\8(\R^n\times[-1,0])} +\|u_t\|_{L^\8(\R^n\sm B_1\times[-1,0])} + \|\cM^\pm_{\cL_0}u_0\|_{L^\8(B_1)} < 1,
\end{align}
and $u$ solves the equation in some large domain $B_{2R}\times(-(2R)^\s,0]$. Note also that by the hypothesis of the theorem and Theorem \ref{regularityint} $u$ is Lipschitz in time in $\R^n\times[-1,0]$ and $C^{1,\alpha^*}(B_{1/2}\times[-1/2,0])$ for some $\a^*$. Hence we only need to prove the estimates in space.

We will prove that there is $\l>0$ and a sequence of functions
\begin{align*}
l_k(x) = a_k +b_k\cdot x,
\end{align*}
such that
\begin{align}
\label{cond1} \sup\limits_{B_{\l^k\times(-\l^{\s k},0)}}|u-l_k|\leq \l^{k(1+\a)},\\
\label{cond2} |a_{k+1}-a_k|\leq C\l^{k(1+\a)},\\
\label{cond3} \l^k|b_{k+1}-b_k|\leq C\l^{k(1+\a)},
\end{align}
for some $\l \in (0,1)$ and $C>0$ universals. That would prove a geometric decay at the origin which implies the desired H\"older modulus of continuity in $B_{1/2}\times(-1/2,0]$ by a covering argument.

Let $l_0=0$. We will construct the sequence by induction. Assume that conditions \eqref{cond1} through \eqref{cond3} hold up some value $k$, we will see how to prove the step $k+1$.

Consider
\begin{align*}
w_k(x,t)=\frac{[u-l_k](\l^k x, \l^{k\s}t)}{\l^{k(1+\a)}}.
\end{align*}
Since $I$ has scale $\s$, then a simple computation tells us that $w_k$ solves an equation of the same ellipticity type with $I_k^{(1)} := I^{(1)}_{\l^{-k(1+\a)},\l^k}$, the rescaled operator, namely
\begin{align*}
(w_k)_t-I_k^{(1)}w_k \leq \l^{k(\s-1-\a)}f_1(\l^k x, \l^{k\s}t)< \eta,
\end{align*}
in an even larger domain ($\l<1$). The other inequality holds for $I^{(2)}_k$, defined similarly, replacing $f_1$ by $f_2$. Since $\s-1-\a\geq\s_0-1-\a>0$ the right hand sides becomes smaller as $k$ increases. Because $I^{(1)}$ and $I^{(2)}$ are close to $I$ at every scale we get that,
\begin{align*}
\left\|I^{(j)}_k-I_k^{(0)}\right\|\leq\left\|I^{(j)}-I^{(0)}\right\|<\eta, \text{ for $j=1,2$}.
\end{align*}
By the inductive hypothesis we also have that $|w_k|\leq 1$ in $B_1\times(-1,0)$, but we only have a control on the tail that deteriorates with every rescaling. Let $\a_1$ such that $\a<\a_1<\min\{\bar \a,\s_0\}$, we will show that we also can construct the sequence $l_k$ so that
\begin{align}\label{cond4}
w_k(x,t)\leq |x|^{1+\a_1}, \text{ in } \1(\R^n\sm B_1)\times[-1,0]\2 \cup \1\R^n\times\{-1\}\2.
\end{align}
Therefore we are also assuming \eqref{cond4} with the set of inductive hypothesis \eqref{cond1}, \eqref{cond2} and \eqref{cond3}. The initial case clearly holds because $u$ is bounded by one.

For $\e>0$ sufficiently small, to be fixed, we chose $\eta$ and $R$ from Theorem \ref{stacomp} so that $w_k(x,t)$ is H\"older continuous in $B_R\sm B_1\times[-1,0]\cup B_R\times\{-1\}$ and the constant does not degenerate with $k$, since $|w_k| < \max(1,|x|^{\alpha_1 + 1})$. Let $h$ such that 
\begin{align*}
h_t-I_k^{(0)}h&=0, \text{ in } B_1\times(-1,0],\\
h(x,t)&=w_k(x,t), \text{ in } \1\R^n\sm B_1\times[-1,0]\2 \cup \1B_1\times\{-1\}\2, 
\end{align*}
and note that thanks to Theorem \ref{stacomp} it satisfies $|w_k-h|<\e$ in $B_1\times[-1,0]$. Since $I^{(0)}_k$ is translation invariant and elliptic with respect to $\cL$, then $h$ has interior $C^{1,\bar \a}$ estimates in space (see \cite{CD}). Note also $w_k$ is Lipschitz in time in $\R^n\sm B_1\times[-1,0]$ and $\cM^{\pm}(w_k(x,-1))$ is bounded in $B_1$, $h$ also has interior $C^{1,\bar \a}$ estimates in time. Let 
\begin{align*}
\bar l_0(x,t) = \bar a +\bar b \cdot x+\bar ct, 
\end{align*}
be the linear part of $h$ at the origin. Since $h\leq 1+\e$ in $B_1\times[-1,0]$, then $\bar a \leq 1+\e$. By the interior estimates we also have that $\bar b, \bar c \leq C_1$, where $C_1$ is a universal constant. Denote now by $\bar l(x)=\bar a +\bar b \cdot x$, we have then
\begin{align*}
|h(x,t)-\bar l(x,t)| \leq C_2\1|x|^{1+\bar \a}+|t|\2,  \text{ in } B_{1/2}\times[-1/2,0],
\end{align*}
and notice that $C_2$ is a universal constant because both the boundary and initial data for $h$ are bounded by $\max(1,|x|^{\alpha_1 + 1})$.

Let $\bar w_k=w_k-\bar l$, we have the following estimates,
\begin{align*}
|\bar w_k(x,t)| &\leq |w_k-h|+|h-\bar l|\\
&\leq \e+C_2(|x|^{1+\bar \a}+|t|),   \text{ in } B_{1/2} \times[-(1/2)^\s,0],\\
|\bar w_k(x,t)| &\leq |w_k|+|\bar l|\\
&\leq 1+(1+\e)+C_1,  \text{ in } \1B_1\sm B_{1/2}\2\times[-(1/2)^\s,0],\\
|\bar w_k(x,t)| &\leq |w_k|+|\bar l|\\
&\leq |x|^{1+\a_1}+(1+\e) +C_1|x|, \text{ in } \1\R^n\sm B_1\2\times[-(1/2)^\s,0].
\end{align*}
Define
\begin{align*}
l_{k+1}(x)&=l_k(x)+\l^{k(1+\a)}\bar l(x/\l^k),\\
w_{k+1}(x,t)&=\frac{[u-l_{k+1}](\l^{k+1}x,\l^{(k+1)\s} t)}{\l^{(k+1)(1+\a)}}=\frac{\bar w_k(\l x, \l^\s t)}{\l^{1+\a}}.
\end{align*}
By taking $\eta$ sufficiently small and $R$ sufficiently large we can assume that $\e$ is such that $\e\leq\l^{1+\bar \a}$. The bounds for $w_{k+1}$ are then
\begin{align*}
|w_{k+1}(x,t)| &\leq \l^{\bar \a-\a}+C_2\l^{\bar \a-\a}|x|^{1+\bar \a}\\
&+C_2\l^{\s-1-\a}|t|, \text{ in } B_{\l^{-1}/2} \times[-(\l^{-1}/2)^\s,0],\\
|w_{k+1}(x,t)|&\leq \l^{\bar \a-\a} +(2+2C_1)\l^{-(1+\a)}, \text{ in } B_{\l^{-1}}\sm B_{\l^{-1}/2}\times[-(\l^{-1}/2)^\s,0],\\
|w_{k+1}(x,t)|&\leq \l^{\bar \a-\a}+\l^{-1-\a}+\l^{\a_1-\a}|x|^{1+\a_1}\\
&+C_1\l^{-\a}|x|, \text{ in } \R^n\sm B_{\l^{-1}}\times[-(\l^{-1}/2)^\s,0].
\end{align*}
Recall that $\s-1-\a>0$, hence we can make $\l^{\bar \a-\a}(1+C_2)+\l^{\s-1-\a}\leq 1$. Therefor we have $w_{k+1}\leq 1$ in $B_{\l^{-1}/2}\times[-(\l^{-1}/2)^\s,0]$, which in particular implies the bounds in the smaller domain $B_1\times[-1,0]$. We also want to check that these bounds allows us to get $|w_{k+1}(x)|\leq |x|^{1+\a_1}$ in $\R^n\sm B_1\times[-1,0]$.

Let $(x,t)\in B_{\l^{-1}/2}\sm B_1\times[-1,0]$, as mentioned before $w_{k+1}(x,t)\leq 1\leq |x|^{1+\a_1}$.
Now let $(x,t)\in B_{\l^{-1}}\sm B_{\l^{-1}/2}\times[-1,0]$, we have
\begin{align*}
|w_{k+1}(x)|&\leq \l^{\bar \a-\a} +(2+2C_1)\l^{-(1+\a)},\\
&\leq (\l^{-1}/2)^{1+\a_1},\\
&\leq |x|^{1+\a_1},
\end{align*}
which holds for $\l$ small enough. Finally let $(x,t)\in\R^n\sm B_{\l^{-1}}\times[-1,0]$, then
\begin{align*}
|w_{k+1}(x)| &\leq \l^{\bar \a-\a}+(\l^{-1-\a}+C_1\l^{-\a})+\l^{\a_1-\a}|x|^{1+\a_1}+C_1\l^{-\a}|x|,\\
&\leq (\l^{-1-\a}+C_1\l^{-\a}) + (1+|x|+|x|^{1+\a_1})/2,\\
&\leq |x|^{1+\a_1},
\end{align*}
again for $\l$ small enough.

Moreover we have
\begin{align*}
|a_{k+1}-a_k| &\leq 2\l^{k(1+\a)}\\
\l^k|b_{k+1}-b_k| &\leq C_1\l^{k(1+\a)}.
\end{align*}
Making $\l$ smaller we can get the desired conditions \eqref{cond2}, \eqref{cond3}.
Finally, we also have that
\begin{align*}
|u(x,t)-l_{k+1}(x)|&=|\l^{(k+1)(1+\a)}w_{k+1}(x/\l^{k+1},t/\l^{(k+1)\s})|\\
&\leq \l^{(k+1)(1+\a)}, \text{ in } B_{\l^{k+1}}\times[-\l^{(k+1)\s},0].
\end{align*}
This completes the proof.
\end{proof}
%

\subsection{Applications}

As in \cite{C2} we proceed to apply Theorem \ref{C1a theorem} to some particular cases. The first case would be an application to a linear operator with variable coefficients and would play the parabolic non local version of the classical Cordes-Nirenberg estimates.

\begin{corollary}\label{apl1}
Let $\s>1$, $f$ a bounded function and let $u$ be a bounded solution of 
\[
u_t-Iu=f, \ \text{in}\ B_1\times[-1,0],
\]
where $I$ is given by
\[
Iu=\int_{\R^n}\d u(x,y)\frac{(2-\s)a(x,y)}{|y|^{n+\s}}dy.
\]
Assume that $u$ is Lipschitz in time in $\R^n\sm B_1\times(-1,0]$, $|\cM^\pm_{\cL_0}u_0|$ is bounded in $B_1$, where $u_0(x)=u(x,-1)$. Suppose also that there is a some small $\eta$ such that the coefficients satisfy 
\[
|a(x,y)-a_0(y)|<\eta, \ \text{in} \ B_1,
\]
where $a_0$ is a bounded function such that the linear operator with kernel $K(y)=(2-\s)a_0(y)/|y|^{n+\s}$ is in $\cL_1$ (see Section \ref{VSP}). Then for any $\a<\s-1$ we have that $u$ is $C^{1,\a}(B_{1/2}\times[-1/2,0])$ in the spatial variable, provided $\eta$ is small enough. 
\end{corollary} 
\begin{proof}
First note that the operator $I$ is translation invariant in time. Now define
\[
I^{(0)}u=(2-\s)\int_{\R^n}\d u(x,y)\frac{a_0(y)}{|y|^{n+\s}}dy,
\]
and note that is translation invariant in space, hence $\partial_t -I^{0}$ has a priori $C^{1,\a}$ estimates (see \cite{CD}). As in \cite{C2}, we can estimate $\left\|I-I^{0}\right\|_\s\leq C\eta$ for some universal constant $C$. Applying Theorem \ref{C1a theorem} we conclude the desired result.
\end{proof}

Another application would be on non linear equations with variable coefficients, using the interior estimates derived in \cite{CD}. A precise example of this would be to consider the non translation invariant operator
\[
Iu=\inf\limits_\a\sup\limits_\b\int_{\R^n}\d u(x,y)\frac{(2-\s)a_{\a\b}(x,y)}{|y|^{n+\s}},
\]
where we ask the coefficients the following conditions
\begin{align*}
|a_{\a\b}-a_0(y)|<\eta, \ \forall \ \a,\b,\\
\L^{-1}\leq a_0(y)\leq \L, \\
|\nabla a_0(y)|\leq C|y|^{-1},
\end{align*}
for a universal $C$ and some bounded function $a_0$. A precise statement of this corollary is the following.
\begin{corollary}\label{apl2}
Assume that $u$ is Lipschitz in time in $\R^n\sm B_1\times(-1,0]$, $|\cM^\pm_{\cL_0}u_0|$ is bounded in $B_1$, where $u_0(x)=u(x,-1)$ and let $\s>\s_0>1$. Let $I^{0}$ be a translation invariant operator in the class $\cL_1$ and let $I$ be a uniformly elliptic operator with respect to $\cL_0$, translation invariant in time. Suppose $u$ solves
\[
u_t-Iu=f \ \text{in} \ B_1\times[-1,0]
\]
for some bounded $f$ and assume that $\left\|I-I^{0}\right\|_\s<\eta$ for every $x\in B_1$ and for some small $\eta$. Then $u$ is $C^{1,\a}(B_{1/2}\times[-1/2,0])$ in space.
\end{corollary}
\begin{remark}
In both Corollaries \ref{apl1} and \ref{apl2} we have also that $u$ is $C^{1,\a}$ in time thanks to the hypothesis on the boundary data.
\end{remark}


\begin{thebibliography}{00}
 
%
 \bibitem{CC}
 L. Caffarelli, X. Cabr\'e. {\em Fully nonlinear elliptic equations.} American Mathematical Society Colloquium Publications, 43. American Mathematical Society, Providence, RI, 1995. vi+104 pp. 
 
\bibitem{CHV}
L. Caffarelli, C. Chan, A. Vasseur. {\em Regularity theory for parabolic nonlinear integral operators.} J. Amer. Math. Soc. 24 (2011), no. 3, 849-869, 45P05.

\bibitem{C1}
L. Caffarelli, L. Silvestre. {\em Regularity theory for fully nonlinear integro differential equations.} Comm. Pure Appl. Math. 62 (2009), no. 5, 597-638.
 
\bibitem{C2}
L. Caffarelli, L. Silvestre. {\em Regularity results for nonlocal equations by approximation} arXiv:0902.4030v2 .
%
\bibitem{C3}
L. Caffarelli, L. Silvestre. {\em The Evans-Krylov theorem for non local fully non linear equations.} Ann. of Math. (2) 174 (2011), no. 2, 1163–1187.
%
\bibitem{C4}
L. Caffarelli, L. Silvestre. {\em On the Evans-Krylov theorem.} arXiv:0905.1336.
%
%
\bibitem{CD}
H. Chang Lara, G. D\'avila. {\em Regularity for solutions of non local parabolic equations.}
%
%
\bibitem{DP}
S. D. Deskmukh, S. R. Pliska. {\em Optimal consumption and exploration of non-renewable resources under uncertainty.} Econometrica, 48 (1980), pp. 177-200.
%
%
\bibitem{Evans}
L. C. Evans. {\em Classical solutions of fully nonlinear, convex, second-order elliptic equations.} Comm. Pure Appl. Math. 35 (1982), 333-363. MR 0649348. Zbl 0469.
%
%
\bibitem{FK}
M. Felsinger, M. Kassmann. {\em Local regularity for parabolic nonlocal operators.}\\ arXiv:1203.2126v1 [math.AP]
%
%
\bibitem{Krylov}
N. V. Krylov. {\em Boundedly inhomogeneous elliptic and parabolic equations.} Izv. Akad. Nauk SSSR Ser. Mat. 46 (1982), 487-523, 670. MR 0661144. Zbl 0529.
35026.
%
%
%
\bibitem{MR}
J. L. Menaldi and M. Robin. {\em Ergodic control of reflected diffusions with jumps.} Appl. Math. Optim. 35 (1997).
%
\bibitem{S2}
L. Silvestre. {\em On the differentiability of the solution to the Hamilton-Jacobi equation with critical fractional diffusion.} Adv. Math. 226 (2011), no. 2, 2020-2039.
%
\bibitem{Soner}
Soner, H. M. {\em Optimal control with state-space constraint.} II. SIAM J. Control Optim. 24 (1986), no. 6, 1110-1122. 
%
\bibitem{Stein}
E. M. Stein. {\em Singular Integrals and Differentiability Properties of Functions.} Princeton Math. Ser. 30, Princeton Univ. Press, Princeton, N.J., 1970.
%
%
 \end{thebibliography}
\end{document}